\documentclass[12pt]{article}

\usepackage{amsmath, epsfig, cite, lineno}
\usepackage{amssymb,amsthm}
\usepackage{amsfonts, color}
\usepackage{latexsym}

\newtheorem{thm}{Theorem}[section]

\newtheorem{cor}[thm]{Corollary}
\newtheorem{conj}[thm]{Conjecture}
\newtheorem{lem}[thm]{Lemma}

\newcommand{\pf}{\noindent{\it Proof.} }

\def\N{{\mathbb N}}
\def\ZZ{{\mathbb Z}}
\def\Z{{\mathbb Z}^+}

\numberwithin{equation}{section}

\begin{document}

%\linenumbers

\begin{center}
{\Large\bf Proof of some conjectures of Z.-W. Sun on\\[5pt] congruences for Ap\'ery polynomials}
\end{center}

\vskip 2mm \centerline{Victor J. W. Guo$^1$  and Jiang Zeng$^{2}$}
\begin{center}
{\footnotesize $^1$Department of Mathematics, East China Normal University,\\ Shanghai 200062,
 People's Republic of China\\
{\tt jwguo@math.ecnu.edu.cn,\quad http://math.ecnu.edu.cn/\textasciitilde{jwguo}}\\[10pt]
%$^2$Universit\'e de Lyon, Lyon, F-69003, France\\
$^2$Universit\'e de Lyon; Universit\'e Lyon 1; Institut Camille
Jordan, UMR 5208 du CNRS;\\ 43, boulevard du 11 novembre 1918,
F-69622 Villeurbanne Cedex, France\\
{\tt zeng@math.univ-lyon1.fr,\quad
http://math.univ-lyon1.fr/\textasciitilde{zeng}} }
\end{center}

%%date: January 4, 2011
%\vskip 5mm
%\noindent {\it Suggested Running title}: Two Identities of Gould

\vskip 0.7cm \noindent{\bf Abstract.}
The Ap\'ery polynomials are defined by $A_n(x)=\sum_{k=0}^{n}{n\choose k}^2{n+k\choose k}^2 x^k$  for all nonnegative integers $n$.
We confirm several conjectures of Z.-W. Sun on the congruences for the sum
$\sum_{k=0}^{n-1}(-1)^k(2k+1) A_k(x)$ with $x\in \ZZ$.
%We also propose some generalizations of these congruences.
%.  Furthermore, we obtain
%\begin{align*}
%\frac{1}{p}\sum_{k=0}^{p-1}(-1)^k(2k+1) A_k(x)  &\equiv  \sum_{k=0}^{p-1}{2k\choose k} x^k\pmod{p^2}
%\end{align*}
%for any odd prime $p$, from which we shall confirm several other conjectures of Z.-W. Sun.
%



\section{Introduction}
The Ap\'ery polynomials \cite{Sun1}  are defined by
\begin{align*}%\label{eq:apery}
A_n(x)=\sum_{k=0}^{n}{n\choose k}^2{n+k\choose k}^2 x^k \qquad (n\in \N).
\end{align*}
Thus  $A_n:=A_n(1)$ are the Ap\'ery numbers \cite{Apery}.
%\begin{align*}%\label{eq:apery}
%A_n=\sum_{k=0}^{n}{n\choose k}^2{n+k\choose k}^2
%=\sum_{k=0}^{n}{n+k\choose 2k}^2{2k\choose k}^2.
%\end{align*}
Many people (see Chowla et al. \cite{CCC}, Gessel \cite{Gessel}, and
Beukers \cite{Beukers}, for example) have studied congruences for Ap\'ery numbers.
Recently,  among other things, Sun \cite{Sun1}
proved   that, for any integer $x$,
\begin{align}
\sum_{k=0}^{n-1}(2k+1)A_k(x) \equiv 0  \pmod n, \label{eq:sun-apery}
\end{align}
and  proposed many  remarkable conjectures. The main objective of this paper is to prove
the following result, which was conjectured by Sun \cite{Sun1}.
\begin{thm}\label{thm:main}
Let $x\in\mathbb{Z}$.
\begin{itemize}
\item[{\rm(i)}] If $n\in\Z$, then
\begin{align}
\sum_{k=0}^{n-1}(-1)^k(2k+1) A_k(x) &\equiv 0 \pmod n. \label{eq:sun1}
\end{align}
\item[{\rm(ii)}]  If $p$ is an odd prime and $\left(\frac{\cdot}{p}\right)$ is  the Legendre symbol modulo $p$, then
\begin{align}
\frac{1}{p}\sum_{k=0}^{p-1}(-1)^k(2k+1) A_k(x) \equiv \left(\frac{1-4x}{p}\right) \pmod{p}. \label{eq:sun2}
\end{align}
%Here $\Z$ denotes the set of positive integers
\item[{\rm(iii)}] If $p>3$ is a prime, then
\begin{align}
\frac{1}{p}\sum_{k=0}^{p-1}(-1)^k(2k+1) A_k &\equiv \left(\frac{p}{3}\right) \pmod{p^2}, \label{eq:sun3}\\
\frac{1}{p}\sum_{k=0}^{p-1}(-1)^k(2k+1) A_k(-2) &\equiv 1- \frac{4}{3}(2^{p-1}-1) \pmod{p^2}.  \label{eq:sun4}
\end{align}
\end{itemize}
%where $q_p(2)$ is the Fermat quotient $(2^{p-1}-1)/p$.
\end{thm}

We shall establish some preliminary
results in Section~2 and prove Theorem~\ref{thm:main} in Section~3.
Then we give some generalizations of the congruences \eqref{eq:sun-apery} and \eqref{eq:sun1} in Section~4.
Sun \cite{Sun1,Sun1-2} also formulated  conjectures for the values of $\sum_{k=0}^{p-1}A_k(x)$ modulo $p^2$ with $x=1,-4,9$.
In particular, he made the following conjecture.
\begin{conj}[Sun \cite{Sun1,Sun1-2}]\label{conj:3}
Let $p$ be an odd prime. Then
\begin{align*}
\sum_{k=0}^{p-1}A_k &\equiv
\begin{cases}
4x^2-2p \pmod{p^2},&\text{if $p\equiv 1,3\pmod 8$ and $p=x^2+2y^2$,} \\
0 \pmod{p^2}, &\text{if $p\equiv 5,7\pmod 8$.}
\end{cases}   %\label{eq:sun5}
\end{align*}
\end{conj}

In an effort to prove the above conjecture, we found a single sum formula for
$\sum_{k=0}^{p-1}A_k(x)$ modulo $p^2$, as given in the following theorem.

\begin{thm}\label{thm:3}
Let $p>3$ be a prime and $x\in\mathbb{Z}$. Then
\begin{align*}
\sum_{k=0}^{p-1}A_k(x)  &\equiv  \sum_{k=0}^{p-1}\frac{(2k)!^4 p}{(4k+1)!k!^4} x^k
\equiv \sum_{k=0}^{(p-1)/2}{p+2k\choose 4k+1}{2k\choose k}^2 x^k \pmod{p^2}.
\end{align*}
\end{thm}

Clearly Theorem \ref{thm:3} implies the following result.
\begin{cor}\label{conj:new}
If $p>3$ is a prime, then  Conjecture \ref{conj:3} is equivalent to the following congruence for binomial sums
\begin{align*}
&\hskip-3mm \sum_{k=0}^{(p-1)/2}{p+2k\choose 4k+1}{2k\choose k}^2 \\[5pt]
&\equiv
\begin{cases}
4x^2-2p \pmod{p^2},&\text{if $p\equiv 1,3\pmod 8$ and $p=x^2+2y^2$,} \\
0 \pmod{p^2}, &\text{if $p\equiv 5,7\pmod 8$.}
\end{cases}   %\label{eq:conjnew}
\end{align*}
\end{cor}

%%%%%%%%%%%%%%%%%%%%%%%%%%%%%%%%%%%%%%%%%%%%%%%%%%%%%%%%%%%%%%%%%%%%%%%%%%%%%%%%%%%%%%%%%%%%%%%%%%%%%
\section{Two preliminary results}
%We first need the following lemma to give another expression of the Ap\'ery polynomials.
\begin{thm}\label{thm:1}
Let $n\in\Z$. Then
\begin{align*}
&\hskip -3mm \frac{1}{n}\sum_{k=0}^{n-1}(-1)^k(2k+1) A_{k}(x) \\
&=(-1)^{n-1}\sum_{m=0}^{n-1} {2m\choose m}x^m
\sum_{k=0}^m {m\choose k}{m+k\choose k} {n-1\choose m+k}{n+m+k\choose m+k}.
\end{align*}
\end{thm}
\begin{proof}
%[Proof of Theorem {\rm\ref{thm:1}.}]
%\noindent{\it Proof of Theorem {\rm\ref{thm:1}.}}
For $\ell,m\in\N$, a special case of the Pfaff-Saalsch\"utz identity (See  \cite[p.~44, Exercise 2.d]{Stanley}
and \cite{Andrews}) reads
\begin{align}
{\ell\choose m}{\ell+m\choose m}
=\sum_{k=0}^m {2m\choose m+k}{\ell-m\choose k}{\ell+m+k\choose k}. \label{eq:lem1}
\end{align}
Multiplying both sides of \eqref{eq:lem1} by ${\ell\choose m}{\ell+m\choose m}$, we obtain another expression of the Ap\'ery polynomials:
\begin{align}
A_\ell(x)=\sum_{m=0}^{\ell}{2m\choose m}x^m\sum_{k=0}^m {m\choose k}{m+k\choose k} {\ell\choose m+k}{\ell+m+k\choose m+k}. \label{eq:square}
\end{align}
Thus,
\begin{align*}
&\hskip -3mm \sum_{\ell =0}^{n-1}(2\ell+1) (-1)^\ell A_{\ell}(x)  \\
&=\sum_{m=0}^{n-1} {2m\choose m}x^m \sum_{k=0}^m {m\choose k} {m+k\choose k}
\sum_{\ell=m+k}^{n-1}(-1)^\ell (2\ell+1){\ell\choose m+k}{\ell+m+k\choose m+k}.
%&=(-1)^{n-1}\sum_{m=0}^{n-1} {2m\choose m}x^m \sum_{k=0}^m n{m\choose k}{m+k\choose k} {n-1\choose m+k}{n+m+k\choose m+k},
\end{align*}
The result then follows by applying the formula
\begin{align}
\sum_{\ell=k}^{n-1}(-1)^\ell (2\ell+1){\ell\choose k} {\ell+k\choose k}=(-1)^{n-1} n{n-1\choose k}{n+k\choose k}, \label{eq:lemsun}
\end{align}
which can be easily verified by induction.
%We also need the second one of the following identities, which can be easily proved by mathematical induction,
%and was already given by  Sun \cite{Sun} (the first identity will be used later).
%
%\begin{lem}\label{lem:sun}
%For all $k, n\in\N$, we have
%\begin{align}
%%\sum_{\ell=k}^{n-1}(2\ell+1){\ell\choose k}{\ell+k\choose k} &=n{n\choose k+1}{n+k\choose k}, \label{eq:lemsun0} \\
%\sum_{\ell=k}^{n-1}(-1)^\ell (2\ell+1){\ell\choose k} {\ell+k\choose k}&=(-1)^{n-1} n{n-1\choose k}{n+k\choose k}. \label{eq:lemsun}
%\end{align}
%\end{lem}
%
 \end{proof}

\begin{thm}\label{thm:2}
Let $p$ be an odd prime and $x\in\mathbb{Z}$. Then
\begin{align*}
\frac{1}{p}\sum_{k=0}^{p-1}(-1)^k(2k+1) A_k(x)  &\equiv  \sum_{k=0}^{p-1}{2k\choose k} x^k\pmod{p^2}.
\end{align*}
\end{thm}

\begin{proof}
%[Proof of Theorem~{\rm\ref{thm:2}.}]
Suppose that $0\leqslant k\leqslant m<p$. If $0\leqslant m+k\leq
p-1$, then
\begin{align*}
{p-1\choose m+k}{p+m+k\choose m+k}=\prod_{i=1}^{m+k}\frac{p^2-i^2}{i^2}\equiv (-1)^{m+k}\pmod{p^2}.
\end{align*}
If $m+k\geqslant p$, then ${p-1\choose m+k}{p+m+k\choose m+k}=0$.
Therefore, by Theorem~{\rm\ref{thm:1}}, we get
\begin{align}
\frac{1}{p}\sum_{\ell =0}^{p-1}(2\ell+1) (-1)^\ell A_{\ell}(x)
\equiv \sum_{m=0}^{p-1} {2m\choose m} x^m \sum_{k=0}^{p-m-1} (-1)^{m+k} {m\choose k}{m+k\choose k}  \pmod{p^2}. \label{eq:congruence1}
\end{align}
Note that, for $0\leqslant k<p-1$ and $p\leqslant m+k<2p$, we have
$$
{2m\choose m}\equiv {m+k\choose k}\equiv 0 \pmod{p},
$$
which means that
\begin{align}
{2m\choose m}{m+k\choose k}\equiv 0\pmod{p^2}. \label{eq:2mm-mk}
\end{align}
Hence, the congruence \eqref{eq:congruence1} may be written as
\begin{align*}
\frac{1}{p}\sum_{\ell =0}^{p-1}(2\ell+1) (-1)^\ell A_{\ell}(x)
&\equiv \sum_{m=0}^{p-1} {2m\choose m} x^m \sum_{k=0}^{m} (-1)^{m+k} {m\choose k}{m+k\choose k}  \\
&=\sum_{m=0}^{p-1}{2m\choose m} x^m  \pmod{p^2},
\end{align*}
where we have used the Chu-Vandermonde summation formula.
\end{proof}

\section{Proof of Theorem \ref{thm:main}}
It is clear that the congruence \eqref{eq:sun1} is an immediate consequence of Theorem \ref{thm:1}. To prove
\eqref{eq:sun2}, we first give the following congruence that was implicitly
given by Sun and Tauraso \cite{ST2} (see also Sun \cite{Sun}).
\begin{lem}\label{lem:3}
Let $p$ be an odd prime and let $a\in\Z$ and $x\in\mathbb{Z}$. Then
\begin{align}
\sum_{k=0}^{p^{a}-1}{2k\choose k}x^k\equiv
\left(\frac{1-4x}{p}\right)^a \pmod p,  \label{eq:prop1}
\end{align}
\end{lem}
\noindent{\it Proof.}  If $p\mid x$, then
\eqref{eq:prop1} clearly holds. If $p\nmid x$, then there is a
positive integer $b$ such that $bx\equiv 1\pmod p$ and
\begin{align*}
\sum_{k=0}^{p^{a}-1}{2k\choose k}x^k \equiv
\sum_{k=0}^{p^{a}-1}{2k\choose k}b^{-k} \pmod p.
\end{align*}
By   \cite[Theorem 1.1]{ST2}, we have
$$
\sum_{k=0}^{p^{a}-1}{2k\choose k}b^{-k} \equiv
\left(\frac{b^2-4b}{p}\right)^a \pmod p.
$$
Since $\left(\frac{x^2}{p}\right)=1$, we obtain
$$
\left(\frac{b^2-4b}{p}\right)=\left(\frac{x^2}{p}\right)\left(\frac{b^2-4b}{p}\right)
=\left(\frac{1-4x}{p}\right).
$$
Combining the above three identities yields \eqref{eq:prop1}. \qed

\medskip
\noindent{\it Proof of \eqref{eq:sun2}.} Theorem \ref{thm:2} implies that
\begin{align}
\sum_{k=0}^{p-1}(-1)^k(2k+1) A_k(x)  &\equiv  p\sum_{k=0}^{p-1}{2k\choose k} x^k\pmod{p^{3}}. \label{eq:2kk1}
\end{align}
The proof then follows from the $a=1$ case of Lemma \ref{lem:3}.  \qed

\medskip
\noindent{\it Proof of \eqref{eq:sun3}.} Letting $x=1$ in \eqref{eq:2kk1}, we have
\begin{align*}
\sum_{k=0}^{p-1}(-1)^k(2k+1) A_k  &\equiv  p\sum_{k=0}^{p-1}{2k\choose k} \pmod{p^{3}}. %\label{eq:2kk1}
\end{align*}
The proof then follows form the $a=1$ case of the congruence \cite[(1.9)]{ST}
\begin{align*}
\sum_{k=0}^{p^a-1}{2k\choose k}\equiv \left(\frac{p^a}{3}\right) \pmod{p^2}.
\end{align*}
Note that the condition $p>3$ is not necessary in this case. \qed

\medskip
\noindent{\it Proof of \eqref{eq:sun4}.}  Letting $x=-2$ in
\eqref{eq:2kk1}, we have
\begin{align}
\sum_{k=0}^{p-1}(-1)^k(2k+1) A_k(-2)  &\equiv  p\sum_{k=0}^{p-1}{2k\choose k}(-2)^k \pmod{p^{3}}. \label{eq:2kk2}
\end{align}
Similarly to \eqref{eq:prop1}, the first congruence in \cite[Theorem 1.1]{Sun} implies that
\begin{align}
\sum_{k=0}^{p^{a}-1}{2k\choose k}x^k\equiv
\left(\frac{1-4x}{p}\right)^a+\left(\frac{1-4x}{p}\right)^{a-1}u_{p-\left(\frac{1-4x}{p}\right)} \pmod{p^2},  \label{eq:ppsun}
\end{align}
where $p\nmid x$ and the sequence $\{u_n\}_{n\geqslant 0}$ is
defined as
$$
u_0=0,\ u_1=1,\ \text{and $u_{n+1}=(b-2)u_n-u_{n-1},\ n\in\Z,$ where $bx\equiv 1\pmod {p^2}$.}
$$
If $b^2-4b\not\equiv 0\pmod{p^2}$, then
$$
u_n=\frac{(b-2+\sqrt{b^2-4b})^n-(b-2-\sqrt{b^2-4b})^n}{2^n\sqrt{b^2-4b}}.
$$
Now suppose that $p>3$. For $a=1$ and $x=-2$, the congruence \eqref{eq:ppsun} reduces to
\begin{align}
\sum_{k=0}^{p-1}{2k\choose k}(-2)^k \equiv 1+u_{p-1}=1-\frac{4^{p-1}-1}{3\cdot 2^{p-2}}
\equiv 1-\frac{4}{3} \left(2^{p-1}-1\right)\pmod {p^2}.  \label{eq:conj4}
\end{align}
Combining \eqref{eq:2kk2} and \eqref{eq:conj4}, we complete the proof. \qed

\medskip
\noindent{\it Remark.} The following stronger version of \eqref{eq:conj4}:
\begin{align*}
\sum_{k=0}^{p-1}{2k\choose k}(-2)^k \equiv 1-\frac{4}{3} \left(2^{p-1}-1\right)\pmod {p^3}
\quad\text{for any prime $p>3$,}
\end{align*}
was independently obtained by Mattarei and Tauraso \cite{MT} and Z.-W. Sun [arxiv:0911.5665v52, Remark after Conjecture A69].

\section{Generalizations of the congruences \eqref{eq:sun-apery} and \eqref{eq:sun1} }
%
%It is natural to generalize the Ap\'ery numbers $A_{n}$ to $\sum_{k=0}^n{n\choose k}^r{n+k\choose k}^r$ for $r\geqslant 2$. In fact,
%these numbers were already considered by Schmidt \cite{Sc} in 1993.
% who proposed the following problem:
%\begin{prob}[Schmidt \cite{Sc}]\label{prob:Sch}
%For any integer $r\geqslant 2$, define a sequence of numbers $\{c_k^{(r)}\}_{k\in \N}$,
%independent of the parameter $n$, by
%\begin{equation*}%\label{eq:zu}
%\sum_{k=0}^n{n\choose k}^r{n+k\choose k}^r
%=\sum_{k=0}^n{n\choose k}{n+k\choose k}c_k^{(r)},
%\end{equation*}
%Is it true that all the numbers $c_k^{(r)}$ are integers?
%\end{prob}
%
%Zudilin~\cite{Zu} confirmed Schmidt's problem by proving that, for all $n,j,r\in\Z$,
%\begin{align}\label{eq:strehl}
%{2n\choose n}^{-1}{2j\choose j}\sum_{k=j}^n (-1)^{n-k}
%\frac{2k+1}{n+k+1}{2n\choose n-k}{k+j\choose k-j}^r\in \Z,
%\end{align}
%which was originally observed by Strehl~\cite{St}. A $q$-analouge of Zudilin's result was later given by
%Guo, Jouhet and Zeng \cite{GJZ}.

Following Schmidt~\cite{Sc} we call the numbers $S_n^{(r)}=\sum_{k=0}^n{n\choose k}^r{n+k\choose k}^r$ the Schmidt numbers,
and define the Schmidt polynomials as follows:
\begin{align*}%\label{eq:apery}
S_n^{(r)}(x)=\sum_{k=0}^{n}{n\choose k}^r{n+k\choose k}^r x^k.
\end{align*}
As generalizations of \eqref{eq:sun-apery} and \eqref{eq:sun1}, we have the following congruences for Schmidt polynomials.
\begin{thm}\label{thm:schmidt}
Let $n\in\Z$, $r\geqslant 2$ and $x\in\mathbb{Z}$. Then
\begin{align}
\sum_{k=0}^{n-1}(2k+1)S_k^{(r)}(x) &\equiv 0  \pmod n, \label{eq:schmidt1} \\
\sum_{k=0}^{n-1}(-1)^k (2k+1)S_k^{(r)}(x) &\equiv 0 \pmod n.  \label{eq:schmidt2}
\end{align}
\end{thm}

\noindent{\it Remark.} When $r=1$, the polynomial $D_n(x):=S_n^{(1)}(x)$ is called the Delannoy polynomial.
Sun \cite[(1.15)]{Sun2} proved that
$$
\frac{1}{n}\sum_{k=0}^{n-1}(2k+1) D_k(x)=\sum_{k=0}^{n-1}{n\choose k+1}{n+k\choose k}x^k\in\mathbb{Z}[x]
\ \text{for $n=1,2,3,\ldots.$}
$$

To prove Theorem 4.1, we need the following Lemma.

\begin{lem}For $\ell, m\in\N$ and $r\geqslant 2$, there exist integers $a_{m,k}^{(r)}$ {\rm(}$m\leqslant k\leqslant rm${\rm)} such that
\begin{align}
{\ell\choose m}^r{\ell+m\choose m}^r=\sum_{k=m}^{rm} a_{m,k}^{(r)}{\ell\choose k}{\ell+k\choose k}. \label{eq:amkr}
\end{align}
\end{lem}
\pf We proceed by mathematical induction on $r$. For $r=2$, the identity \eqref{eq:square} gives
$$
a_{m,m+k}^{(2)}={2m\choose m}{m\choose k}{m+k\choose k},\ 0\leqslant
k\leqslant m.
$$
Suppose that \eqref{eq:amkr} holds for $r$. A special case of the Pfaff-Saalsch\"utz identity \cite[(1.4)]{Andrews} reads
\begin{align*}
{\ell+m\choose m+n}{\ell+n\choose n}
=\sum_{k=0}^n {\ell-m\choose k} {m\choose k+m-n}{\ell+m+k\choose \ell}, %\label{eq:lem3}
\end{align*}
which can be rewritten as
\begin{align}
{\ell\choose n}{\ell+n\choose n}
=\sum_{k=0}^n \frac{(m+n)!k!}{(m+k)!n!}{m\choose n-k}{\ell-m\choose k}{\ell+m+k\choose k}. \label{eq:lem03}
\end{align}

Now, multiplying both sides of \eqref{eq:amkr} by
${\ell\choose m}{\ell+m\choose m}$ and applying \eqref{eq:lem03}, we have
\begin{align*}
{\ell\choose m}^{r+1}{\ell+m\choose m}^{r+1}
&=\sum_{k=m}^{rm} a_{m,k}^{(r)}{\ell\choose k}{\ell+k\choose k}{\ell\choose m}{\ell+m\choose m} \\
&=\sum_{k=m}^{rm}\sum_{i=0}^k a_{m,k}^{(r)}
\frac{(m+k)!i!}{(m+i)!k!}{m\choose k-i}{\ell-m\choose i}{\ell+m+i\choose i}{\ell\choose m}{\ell+m\choose m}\\
&=\sum_{k=m}^{rm}\sum_{i=0}^k a_{m,k}^{(r)}{m+k\choose k}{m\choose
k-i}{m+i\choose i}{\ell\choose m+i}{\ell+m+i\choose m+i},
\end{align*}
which implies that
$$
a_{m,m+i}^{(r+1)}=\sum_{k=m}^{rm}{m+k\choose k}{m\choose
k-i}{m+i\choose i} a_{m,k}^{(r)}\in\mathbb{Z},\ 0\leqslant
i\leqslant rm.
$$
The proof of the inductive step is then completed. \qed

\medskip
\noindent{\it Proof of Theorem \ref{thm:schmidt}.}
Applying \eqref{eq:amkr} and the easily checked formula (see \cite[Lemma 2.1]{Sun})
$$
\sum_{\ell=k}^{n-1}(2\ell+1){\ell\choose k}{\ell+k\choose k} =n{n\choose k+1}{n+k\choose k},
$$
we have
\begin{align*}
\sum_{\ell=0}^{n-1}(2\ell+1) S_{\ell}^{(r)}(x)
&=\sum_{\ell=0}^{n-1}(2\ell+1) \sum_{m=0}^\ell {\ell\choose m}^r{\ell+m\choose m}^r x^m \\
&=\sum_{m=0}^{n-1} x^m \sum_{\ell=m}^{n-1} \sum_{k=m}^{rm} a_{m,k}^{(r)}(2\ell+1){\ell\choose k}{\ell+k\choose k}\\
&=n\sum_{m=0}^{n-1} x^m \sum_{k=m}^{rm} a_{m,k}^{(r)}{n\choose k+1}{n+k\choose k}.
\end{align*}
Similarly, applying \eqref{eq:amkr} and \eqref{eq:lemsun}, we have
\begin{align*}
\sum_{\ell=0}^{n-1}(-1)^\ell (2\ell+1) S_{\ell}^{(r)}(x)
=(-1)^{n-1}n\sum_{m=0}^{n-1} x^m \sum_{k=m}^{rm} a_{m,k}^{(r)}{n-1\choose k}{n+k\choose k}.
\end{align*}
Therefore the congruences \eqref{eq:schmidt1} and \eqref{eq:schmidt2} hold. \qed

Furthermore, similarly to the proof of \cite[Theoreem 1.2]{GZ}, we can prove the following generalization of Theorem \ref{thm:schmidt}.
\begin{thm}\label{thm:4}
Let $n\in\Z$, $a\in\N$, $r\geqslant 2$ and $x\in\mathbb{Z}$. Then
\begin{align*}
\sum_{k=0}^{n-1}\varepsilon^k (2k+1)k^a (k+1)^a S_k^{(r)}(x) &\equiv 0  \pmod n, \\
\sum_{k=0}^{n-1}\varepsilon^k (2k+1)^{2a+1} S_k^{(r)}(x) &\equiv 0  \pmod n,
\end{align*}
where $\varepsilon=\pm1$.
\end{thm}

%It should be mentioned that Sun \cite{Sun} has made the following stronger conjecture instead of \eqref{eq:sun1}.
%\begin{conj}[Sun] \label{conj:3}
%For any $\varepsilon\in\{\pm1\}$, $m,n\in\Z$ and $x\in\mathbb{Z}$, we have
%\begin{align}
%\sum_{k=0}^{n-1}(2k+1) \varepsilon^k A_k(x)^m &\equiv 0 \pmod n.
%\end{align}
%\end{conj}
%However, we are unable to solve this conjecture for $m\geqslant 2$. Nevertheless, we think that Conjecture \ref{conj:3}
%can be further generalized to Schmidt polynomials as follows:
We conclude this section with the following conjecture, which is due to Sun~\cite{Sun}  in the case $r=2$.
\begin{conj}
Let $\varepsilon\in\{\pm1\}$, $m,n\in\Z$, $r\geqslant 2$, and
$x\in\mathbb{Z}$. Then
\begin{align*} %\label{eq:conj}
\sum_{k=0}^{n-1}(2k+1) \varepsilon^k S_k^{(r)}(x)^m &\equiv 0 \pmod n.
\end{align*}
\end{conj}

\section{Proof of Theorem \ref{thm:3}}
From  the Lagrange interpolation formula for the polynomial
$\prod_{k=1}^m(x-k)$ at the values  $-i$ ($0\leqslant i\leqslant m$)
of $x$ we immediately derive that
\begin{align}
\sum_{k=0}^m {m\choose k}{m+k\choose k}\frac{(-1)^{m-k}}{x+k}=\frac{1}{x}\prod_{k=1}^m\frac{x-k}{x+k}. \label{eq:4m+1x}
\end{align}
A proof of  \eqref{eq:4m+1x} by using the creative telescoping method was given in \cite[Lemma 3.1]{Mortenson}.

\medskip
\noindent{\it Remark.} The identity \eqref{eq:4m+1x} plays an
important role in Mortenson's proof \cite{Mortenson} of
Rodriguez-Villegas's conjectures \cite{RV} on supercongruences for
hypergeometric Calabi-Yau manifolds of dimension $d\leqslant 3$.
Letting $x=m+\frac{1}{2}$ in \eqref{eq:4m+1x}, we obtain
 \begin{align}
\sum_{k=0}^m {m\choose k}{m+k\choose k}\frac{(-1)^{m-k}}{2m+2k+1}=\frac{(2m)!^3}{(4m+1)!m!^2}. \label{eq:4m+1}
\end{align}
\medskip

Let $p>3$ be a prime.
Similarly to the proof of Theorem~\ref{thm:1}, applying \eqref{eq:square} and
the trivial identity
$$
\sum_{\ell=m}^{p-1}{\ell\choose m+k}{\ell+m+k\choose m+k}={2m+2k\choose m+k}{p+m+k\choose 2m+2k+1},
$$
we have
\begin{align}
\sum_{\ell =0}^{p-1} A_{\ell}(x)
&=\sum_{m=0}^{p-1} {2m\choose m}x^m\sum_{k=0}^m {m\choose k}{m+k\choose k}
{2m+2k\choose m+k}{p+m+k\choose 2m+2k+1}. \label{eq:sumax}
\end{align}
Note that, if $0\leqslant m+k<p$, then
\begin{align*}
{2m+2k\choose m+k}{p+m+k\choose 2m+2k+1} &=\frac{(2m+2k)!p \prod_{i=1}^{m+k}(p^2-i^2)}{(m+k)!^2(2m+2k+1)!} \\
&\equiv (-1)^{m+k}\frac{p}{2m+2k+1} \pmod {p^2},
\end{align*}
while if $m+k\geqslant p$, then ${2m+2k\choose m+k}{p+m+k\choose
2m+2k+1}=0$.

From \eqref{eq:sumax}, we deduce that
\begin{align}
\sum_{\ell =0}^{p-1} A_{\ell}(x)
&\equiv \sum_{m=0}^{p-1} {2m\choose m}x^m\sum_{k=0}^{p-m-1} {m\choose k}{m+k\choose k}\frac{(-1)^{m+k}p}{2m+2k+1} \pmod{p^2}. \label{eq:congfrac}
\end{align}
Since $p>3$, we have $2m+2k+1\not\equiv 0\pmod{p^2}$ for any
$0\leqslant k\leqslant m\leqslant p-1$. Thus, by \eqref{eq:2mm-mk}
and \eqref{eq:4m+1}, the congruence \eqref{eq:congfrac} is
equivalent to
\begin{align}
\sum_{\ell =0}^{p-1} A_{\ell}(x)
&\equiv \sum_{m=0}^{p-1} {2m\choose m}x^m\sum_{k=0}^{m} {m\choose k}{m+k\choose k}\frac{(-1)^{m+k}p}{2m+2k+1} \nonumber\\
&=\sum_{m=0}^{p-1} \frac{(2m)!^4 p}{(4m+1)!m!^4} x^m \pmod{p^2}.  \label{eq:4m+1two}
\end{align}
On the other hand, it is not difficult to see that
$$
\frac{(2m)!^4 p}{(4m+1)!m!^4}\equiv 0 \pmod{p^2}\quad\text{for $p/2<m<p$,}
$$
and
$$
\frac{p(2m)!^2}{(4m+1)!}\equiv
\frac{p\prod_{i=1}^{2m}(p^2-i^2)}{(4m+1)!} ={p+2m\choose
4m+1}\pmod{p^2}\quad\text{for $p>3$ and $0\leqslant m<p/2$.}
$$
This proves that
\begin{align}
\sum_{m=0}^{p-1} \frac{(2m)!^4 p}{(4m+1)!m!^4} x^m
&\equiv\sum_{m=0}^{(p-1)/2} \frac{(2m)!^4 p}{(4m+1)!m!^4} x^m  \nonumber \\
&\equiv\sum_{m=0}^{(p-1)/2}{p+2m\choose 4m+1}{2m\choose m}^2 x^m \pmod{p^2}.   \label{eq:4m+1three}
\end{align}
Combining \eqref{eq:4m+1two} and \eqref{eq:4m+1three}, we complete the proof.
%It is easy to check that Theorem {\rm\ref{thm:3} is not true for $p=3$.

%%%%%%%%%%%%%%%%%%%%%%%%%%%%%%%%%%%%%%%%%%%%%%%%%%%%%%%%%%%%%%%%%%%%%%%%%%%%%%%%%%%%%%%%%%%%%%%%%%%%%%

\vskip 5mm
\noindent{\bf Acknowledgments.} This work was partially
supported by the Fundamental Research Funds for the Central Universities,  Shanghai Rising-Star Program (\#09QA1401700),
Shanghai Leading Academic Discipline Project (\#B407), and the National Science Foundation of China (\#10801054).


\begin{thebibliography}{99}
\small \setlength{\itemsep}{-.8mm}
\bibitem{Andrews}G.E. Andrews, Identities in combinatorics, I: On sorting two ordered sets, Discrete Math. 11 (1975), 97--106.

\bibitem{Apery}R. Ap\'ery, Irrationalit\'e de $\zeta(2)$ et $\zeta(3)$, Ast\'erisque 61 (1979), 11--13.

\bibitem{Beukers}F. Beukers, Another congruence for the Ap\'ery numbers, J. Number Theory 25
(1987), 201--210.

%\bibitem{BD}S. Brunetti and A. Del Lungo, On the polynomial
%$\frac{1}{[n]_q}{n\brack k}_q$, Adv. Appl. Math. 33 (2004), 487--491.

%\bibitem{CHV}J.S. Caughman, C.R. Haithcock and J.J.P. Veerman, A note on lattice chains
%and Delannoy numbers, Discrete Math. 308 (2008), 2623--2628.

\bibitem{CCC}S. Chowla, J. Cowles and M. Cowles, Congruence properties of Ap\'ery numbers,
J. Number Theory 12 (1980), 188-190.

%\bibitem{De} J. D\'esarm\'enien,  Un analogue des congruences de Kummer pour les $q$-nombres d'Euler,
%European J. Combin. {3} (1982), 19--28.

\bibitem{Gessel}I. Gessel, Some congruences for Ap\'ery Numbers, J. Number Theory 14 (1982), 362--368.

\bibitem{GZ}V.J.W. Guo and J. Zeng, New congruences for sums involving Apery numbers or central Delannoy numbers,
preprint, arXiv:1008.2894.

%\bibitem{GJZ}V.J.W. Guo, F. Jouhet, and J. Zeng,
%Factors of alternating sums of products of binomial and $q$-binomial coefficients, Acta Arith. 127 (2007), 17--31.

\bibitem{MT}S. Mattarei and R. Tauraso, Congruences for central binomial sums and finite polylogarithms,
preprint, arXiv:1012.1308.


\bibitem{Mortenson}E. Mortenson, Supercongruences between truncated $_2F_1$ by geometric functions
and their Gaussian analogs, Trans. Amer. Math. Soc. 355 (2003), 987--1007.

%\bibitem{PWZ}M. Petkovsek, H. Wilf, and D. Zeilberger, A=B, A. K. Peters, Ltd., Wellesley, MA, 1996.

\bibitem{RV}F. Rodriguez-Villegas, Hypergeometric families of Calabi-Yau manifolds, in:
Calabi-Yau Varieties and Mirror Symmetry (Toronto, ON, 2001), pp. 223-231,
Fields Inst. Commun., 38, Amer. Math. Soc., Providence, RI, 2003.

\bibitem{Sc}A.L. Schmidt, Generalized $q$-Legendre polynomials,
J. Comput. Appl. Math. 49 (1993), 243--249.

\bibitem{Stanley}R.P. Stanley, Enumerative Combinatorics, Vol. I, Cambridge University Press, Cambridge, 1997.

%\bibitem{St}V. Strehl, Binomial identities --- combinatorial and algorithmical aspects,
% Discrete Math. 136 (1994), 309--346.

\bibitem{Sun}Z.-W. Sun, Binomial coefficients, Catalan numbers and Lucas quotients, Sci. China Math. 53 (2010), 2473--2488.

%\bibitem{Sun1}Z.-W. Sun, On Ap\'ery numbers and generalized central trinomial coefficients, preprint, arXiv:1006.2776.

%\bibitem{Sun2}Z.-W. Sun, On Ap\'ery numbers and generalized central trinomial coefficients (II), preprint, arXiv:1008.3887.

\bibitem{Sun1}Z.-W. Sun, On sums of Ap\'ery polynomials and related congruences, preprint, arXiv:1101.1946.

\bibitem{Sun1-2}Z.-W. Sun, Conjectures and results on $x^2$ mod $p^2$ with $4p=x^2+dy^2$, in:
Number Theory and the Related Topics, Eds., Y. Ouyang, C. Xing, F. Xu and
P. Zhang, Higher Education Press \& International Press, Beijing and Boston, in
press, online version is available from {\tt http://arxiv.org/abs/1103.4325}

\bibitem{Sun2}Z.-W. Sun, Congruences involving generalized trinomial coefficients, preprint, arXiv:1008.3887.

\bibitem{ST2}Z.-W. Sun, R. Tauraso, New congruences for central binomial coefficients,
 Adv. Appl. Math. 45 (2010), 125--148.

\bibitem{ST}Z.-W. Sun and R. Tauraso, On some new congruences for binomial coefficients,
Int. J. Number Theory 7 (2011), 645--662.




%\bibitem{Sunopen}Z.-W. Sun, Open conjectures on congruences, preprint, arXiv:0911.5665v39.

%\bibitem{Zu}W. Zudilin, On a combinatorial problem of Asmus Schmidt, Electron. J. Combin. 11 (2004), \#R22.



\end{thebibliography}
\end{document}